\documentclass[10pt]{amsart}

\usepackage{amsmath}
\usepackage{amsthm}
\usepackage{amssymb}
\usepackage{enumerate}
\usepackage{fullpage}
\usepackage[all]{xy}
\usepackage{hyperref}
\usepackage{url}

\usepackage{graphicx}
\usepackage{pdflscape}

\usepackage{float}
\floatstyle{boxed}
\restylefloat{figure}

\usepackage{color}

\usepackage{bm}

\DeclareFontFamily{OMS}{smallo}{}
\DeclareFontShape{OMS}{smallo}{m}{n}{<->s*[.65]cmsy10}{}
\DeclareSymbolFont{smallo@m}{OMS}{smallo}{m}{n}
\DeclareMathSymbol{\smallo}{\mathord}{smallo@m}{79}

\theoremstyle{definition}
\newtheorem{thm}{Theorem}[section]
\newtheorem*{thmunnumbered}{Theorem}
\newtheorem{distalcriterion}[thm]{Distal Criterion}
\newtheorem{theorem}{Theorem}[section]
\newtheorem{definition}[thm]{Definition}
\newtheorem{lemma}[thm]{Lemma}

\newtheorem{cor}[thm]{Corollary}
\newtheorem{prop}[thm]{Proposition}

\newtheorem{remark}[thm]{Remark}

\newtheorem{question}[thm]{Question}
\newtheorem{example}[thm]{Example}

\newtheorem{G}[thm]{G}

\newcommand\N{\mathbb{N}}

\newcommand\Q{\mathbb{Q}}
\newcommand\R{\mathbb{R}}
\newcommand\M{\mathbb{M}}

\DeclareMathOperator\Th{Th}

\renewcommand\epsilon{\varepsilon}

\DeclareFontFamily{U}{fsy}{}
\DeclareFontShape{U}{fsy}{m}{n}{<->s*[.9]psyr}{}
\DeclareSymbolFont{der@m}{U}{fsy}{m}{n}
\DeclareMathSymbol{\der}{\mathord}{der@m}{182}

\renewcommand\L{\mathcal{L}}

\author{Allen Gehret}
\author{Travis Nell}
\title{Hamel spaces and distal expansions}
\email{allen@math.ucla.edu}
\email{tnell2@illinois.edu}
\address{Department of Mathematics, University of California, Los Angeles, Los Angeles, CA 90095}
\address{Department of Mathematics, University of Illinois at Urbana-Champaign, Urbana, IL 61801}
\date{\today}

\begin{document}

\maketitle

\setcounter{tocdepth}{1}
\tableofcontents

\begin{abstract}
In this note, we construct a distal expansion for the structure $(\R; +,<,H)$, where $H\subseteq \R$ is a dense $\Q$-vector space basis of $\R$ (a so-called \emph{Hamel basis}). Our construction is also an expansion of the dense pair $(\R; +,<,\Q)$ and has full quantifier elimination in a natural language.
\end{abstract}

\section{Introduction}

\noindent
Distal theories were introduced by Simon in~\cite{Distal} as a way to isolate those NIP theories which are purely unstable.
For example, all $o$-minimal theories and all $P$-minimal theories are distal, whereas the theory of algebraically closed valued fields is non-distal due to the presence of a stable algebraically closed residue field.
Distality has useful combinatorial consequences. For instance, in~\cite{ChernikovStarchenko},
Chernikov and Starchenko show that distal theories enjoy a version of
the \emph{strong Erd\"{o}s-Hajnal Property}. 
These combinatorial consequences also apply to theories which have a distal expansion, i.e., to reducts of a distal theory.
This paper shows that a certain class of non-distal theories have a distal expansion. In this introduction, we describe the most important case of this construction.

%
%
%
%

\medskip\noindent
In~\cite{Hieronymi_Nell}, Hieronymi and Nell showed that two particular structures are not distal: $(\R;+,<,\Q)$ and $(\R;+,<,H)$, where $H\subseteq \R$ is a dense $\Q$-vector space basis of $\R$ (a so-called \emph{Hamel basis}). 
These findings were initially unexpected, as these structures are closely related to their common o-minimal reduct $(\R; +,<)$.
Simon then asked whether these structures at least have some distal expansion. 
In~\cite{NellMLQ}, Nell constructed a distal expansion of $(\R;+,<,\Q)$, essentially by equipping the quotient $\R/\Q$ with a linear order.
For the structure $(\R;+,<,H)$, a similar trick could not be employed since this structure has elimination of imaginaries~\cite{DolichMillerSteinhorn}.
In this paper we construct a distal expansion of $(\R;+,<,H)$. We now describe our construction, as it applies to $(\R;+,<,H)$, in greater detail.

\medskip\noindent
We expand $(\R;+,<,H)$ to a structure $(\R;+,<,H,v,<_1,\infty)$ with three new primitives: a unary function $v$, a binary relation $<_1$, and a constant $\infty$. 

\medskip\noindent
First, we define $v:\R\setminus\{0\}\to H\subseteq \R$ as follows. Let $\alpha\in\R\setminus\{0\}$ be nonzero.
As $H$ is a basis of $\R$ as a vector space over $\Q$, there are $n\geq1$, basis vectors $h_1,\ldots,h_n\in H$ and scalars $q_1,\ldots,q_n\in \Q^{\times}$ such that $\alpha = q_1h_1+\cdots+q_nh_n$ and $h_1<\cdots<h_n$. The data $(n,h_1,\ldots,h_n,q_1,\ldots,q_n)$ is uniquely determined by the requirement that $h_1<\cdots<h_n$. Thus we define $v(\alpha):= h_1$, and this will be well-defined.

\medskip\noindent
We define the binary relation $<_1$ as the unique ordering which makes $(\R;+)$ into an ordered group such that
\[
0 <_1  \alpha \ :\Longleftrightarrow \ 0  <  q_1,\quad\text{for every $\alpha=q_1h_1+\cdots+q_nh_n\neq 0$ as above.}
\]

\medskip\noindent 
Finally, we introduce a new element $\infty$ as a default value:
\[
v(0) \ = \ v(\infty) \ = \ \alpha+\infty \ = \ \infty+\alpha \ = \ \infty +\infty \ := \ \infty, \quad  \alpha <\infty,\quad\text{and}\quad \alpha  <_1  \infty
\]
for every $\alpha\in\R$. We introduce this element $\infty$ mainly for aesthetic reasons. Indeed, this function $v:\R\to H\cup\{\infty\}$ is in fact a convex valuation on the ordered abelian group $(\R;+,<_1)$ with value set $(H\cup\{\infty\},<)$.

\medskip\noindent
In this paper, we work over an arbitrary ordered $C$-vector space, where $C$ is an ordered field. We call the valuation $v$ constructed as above a \emph{Hamel valuation}, as its image is a Hamel basis of the underlying ordered vector space. 
Likewise, we call the resulting structure a \emph{Hamel space}. In a natural language $\mathcal{L}_{\text{Ham}}$, we formulate a certain theory $T_{\text{Ham}}$ of these Hamel spaces. The main result of the paper is:

\begin{thmunnumbered}[\ref{THamdistal}]
$T_{\operatorname{Ham}}$ is distal.
\end{thmunnumbered}

\medskip\noindent
In \S\ref{distalNIPsection} we recall the definition of distality for a theory $T$, as well as for a partitioned formula $\varphi(x;y)$. We also provide the statement of Distal Criterion~\ref{distal_multi}, a criterion we will later use to prove that the theory $T_{\text{Ham}}$ is distal.

\medskip\noindent
In \S\ref{IorderedCvectorspaces} we introduce \emph{$I$-ordered $C$-vector spaces}. These are vector spaces over an ordered field $C$, equipped with a family of compatible linear orderings $<_i$, indexed by $i\in I$. We construct a complete theory $T_{C,I}$ of these spaces, and show that it is distal. The arguments in this section are routine; however it establishes the language and theory that we will expand when constructing $T_{\text{Ham}}$.

\medskip\noindent
In \S\ref{Hamel spaces} we introduce \emph{Hamel valuations} and \emph{Hamel spaces over $C$}. We give a language $\mathcal{L}_{\text{Ham}}$ of Hamel spaces over $C$, and prove that a certain $\mathcal{L}_{Ham}$-theory $T_{\text{Ham}}$ admits quantifier elimination and is complete. This section is the main technical part of the paper.
%
In \S\ref{DistalityHamelSpacesSection} we prove the main result of the paper. This involves various lemmas on how indiscernible sequences behave in Hamel spaces and applying Distal Criterion~\ref{distal_multi}. We also point out several consequences of distality which are also of model-theoretic interest.

\medskip\noindent
In \S\ref{ConnectionDensePairs} we show in what sense models of $T_{\text{Ham}}$ are expansions of an o-minimal structure with a dense independent set (like $(\R;+,<,H)$). As a bonus, our models of $T_{\text{Ham}}$ are also expansions of dense pairs (like $(\R;+,<,\Q)$), and we show how this works as well. We also point out some natural follow-up questions.
%
Finally, in \S\ref{DPranksection} we show that $T_{\text{Ham}}$ is not strongly dependent, in contrast to the theories of dense independent sets or dense pairs.

\subsection*{Conventions} Throughout, $m$ and $n$ range over $\N = \{0,1,2,\ldots\}$.

\medskip\noindent
 All orderings are total. Let $S$ be a set and let $<$ be an ordering on $S$. We say that a subset $P\subseteq S$ is \textbf{downward closed}, or is a \textbf{cut}, if for all $a,b\in S$, if $b\in P$ and $a<b$, then $a\in P$. A \textbf{well-indexed sequence} is a sequence $(a_{\rho})$ whose terms $a_{\rho}$ are indexed by the elements $\rho$ of an infinite well-ordered set without a greatest element.

\medskip\noindent Let $I$ be an index set. An \textbf{$I$-ordering} on $S$ is a family $(<_i)_{i\in I}$ of orderings on $S$, and $S$ equipped with this family is referred to as an \textbf{$I$-ordered set}. Suppose $S$ is equipped with an $I$-ordering $(<_i)_{i\in I}$. A subset of $S$ is viewed as an $I$-ordered set as well by the induced ordering for each $i$. We put $S_{\infty}:= S\cup\{\infty\}$, $\infty\not\in S$, with $I$-ordering on $S$ extended to an $I$-ordering on $S_{\infty}$ by declaring $S<_i \infty$ for each $i\in I$. Occasionally, we take two distinct elements $-\infty,\infty\not\in S$ and extend the $I$-ordering on $S$ to $S_{\pm\infty}:= S\cup\{-\infty,\infty\}$ by declaring $-\infty<_iS<_i\infty$ for each $i\in I$. A \textbf{polycut} in $S$ is a family $\big((P_i,Q_i)\big)_{i\in I}$ of pairs of subsets of $S$ such that for each $i\in I$, $P_i$ is a downward closed subset of $(S,<_i)$ and $Q_i = S\setminus P_i$. Given an element $b$ in an $I$-ordered set extending $S$, we say that $b$ \textbf{realizes} the polycut $\big((P_i,Q_i)\big)_{i\in I}$ if for every $i\in I$, $P_i<_i b<_i Q_i$. Given $i\in I$ and $a,b\in S_{\pm\infty}$ such that $a<_ib$, we define
\[
(a,b)_i \ := \ \{s\in S: a<_i s<_i b\}.
\]
Suppose $i\in I$ and $A$ is a finite subset of $S$. Then by $\min_iA$ we mean $\min A$ with respect to the linear order $(S,<_i)$.

\medskip\noindent
If $G$ is an expansion of an additively written abelian group, then we set $G^{\neq}:= G\setminus\{0\}$. For a field $C$
 we let $C^{\times}:= C\setminus \{0\} = C^{\neq}$ be its multiplicative group of units.

\medskip\noindent
In general we adopt the model theoretic conventions of Appendix B of~\cite{ADAMTT}. In particular, $\L$ can be a many-sorted language. For a complete $\L$-theory $T$, we will sometimes consider a model $\M\models T$ and a cardinal $\kappa(\M)>|\L|$ such that $\M$ is $\kappa(\M)$-saturated and every reduct of $\M$ is strongly $\kappa(\M)$-homogeneous. Such a model is called a \textbf{monster model} of $T$. In particular, every model of $T$ of size $\leq\kappa(\M)$ has an elementary embedding into $\M$. All variables are finite multivariables. By convention we will write ``indiscernible sequence'' when we mean ``$\emptyset$-indiscernible sequence''.

\section{Preliminaries on distality}
\label{distalNIPsection}

\noindent
\emph{Throughout this section $\L$ is a language and $T$ is a complete $\L$-theory.}

\subsection*{Definition of distality}
\emph{In this subsection we fix a monster model $\M$ of $T$}.

\begin{definition}
\label{distaldef}
We say that $T$ is \textbf{distal} if for every $A\subseteq \M$, for every  $x$, and for every indiscernible sequence $(a_i)_{i\in I}$  from $\M_x$, if
\begin{enumerate}
\item $I = I_1+(c)+I_2$ where $I_1$ and $I_2$ are infinite, and
\item $(a_i)_{i\in I_1+I_2}$ is $A$-indiscernible,
\end{enumerate}
then $(a_i)_{i\in I}$ is $A$-indiscernible. 
Furthermore, we say that an $\L$-structure $\bm{M}$ is \textbf{distal} if $\Th(\bm{M})$ is distal.
\end{definition}

\noindent
It is also convenient to define what it means for a formula $\varphi(x;y)$ to be distal:

\begin{definition}
\label{distaldefformula}
We say a formula $\varphi(x;y)$ is \textbf{distal}
if for every $b\in\M_y$ and every indiscernible sequence $(a_i)_{i\in I}$ from $\M_{x}$, if
\begin{enumerate}
\item $I = I_1+(c)+I_2$ where $I_1$ and $I_2$ are infinite, and
\item $(a_i)_{i\in I_1+I_2}$ is $b$-indiscernible,
\end{enumerate}
then 
$
\models \varphi(a_c;b)\leftrightarrow \varphi(a_i;b)
$
for every $i\in I$.
\end{definition}

\begin{remark}
The collection of all distal formulas in the variables $(x;y)$ is closed under arbitrary boolean combinations, including negations. In the literature, there is another local notion of distality: that of a formula $\varphi(x;y)$ having a \emph{strong honest definition} (see~\cite[Theorem 21]{ExternallyDefinableII}). The collection of formulas $\varphi(x;y)$ which have a strong honest definition is in general only closed under \emph{positive} boolean combinations.
\end{remark}

\begin{lemma}
\label{equivdistaldefs}
The following are equivalent:
\begin{enumerate}
\item $T$ is distal;
\item every $\varphi(x;y)\in\L$ is distal.
\end{enumerate}
\end{lemma}

\begin{example}\label{o-mindist}
All o-minimal theories are distal; see~\cite[Lemma 2.10]{Distal}. In particular, the theory of ordered vector spaces over an ordered field $C$ is distal.
\end{example}

\subsection*{A distal criterion}
To set the stage for Distal Criterion~\ref{distal_multi} below, we now consider an extension $\L(\mathfrak{f}):=\ \L\cup\{\mathfrak{f}\}$ of the language $\L$ by a new unary function symbol $\frak{f}$ involving sorts which are already present in $\L$. We also consider $T(\mathfrak{f})$, a complete $\L(\mathfrak{f})$-theory extending $T$. Given a model $\bm{M}\models T$ we denote by $(\bm{M},\mathfrak{f})$ an expansion of $\bm{M}$ to a model of $T(\mathfrak{f})$.
For a subset $X$ of a model $\bm{M}$, we let $\langle X \rangle$ denote the $\L$-substructure of $\bm{M}$ generated by $X$. If $\bm{M}$ is a submodel of $\bm{N}$, we let $\bm{M}\langle X \rangle$ denote $\langle M \cup X \rangle\subseteq\bm{N}$. For this subsection we also fix a monster model $\M$ of $T(\mathfrak{f})$. Note that $\M\!\upharpoonright\!\L$ is then a monster model of $T$.

\medskip\noindent
Distal Criterion~\ref{distal_multi} is a many-sorted generalization of~\cite[2.1]{Hieronymi_Nell}. The version we give below is a consequence of~\cite[2.6]{GehretKaplan} \emph{In the statement of~\ref{distal_multi},  $x,x',y,z$  are variables.}

\begin{distalcriterion}
\label{distal_multi}
Suppose $T$ is a distal theory and the following conditions hold:
\begin{enumerate}
\item The theory $T(\mathfrak{f})$ has quantifier elimination.
\item For every model $(\bm{N},\mathfrak{f}) \models T(\mathfrak{f})$, every  substructure $\bm{M} \subseteq \bm{N}$ such that $\frak{f}(M) \subseteq M$, every $x$, and every $c \in N_x$, there is a $y$ and $d \in \frak{f}\big(\bm{M}\langle c\rangle\big)_y$ such that
\[
\frak{f}\big(\bm{M}\langle c\rangle\big)  \subseteq \big\langle \frak{f}(M),d \big\rangle.
\]
\item Suppose that $x'$ is an initial segment of $x$,  $g,h$ are $\L$-terms of arities $xy$ and $x'z$ respectively, $b_1\in \M_y$, and $b_2\in\frak{f}(\M)_z$. If $(a_i)_{i\in I}$ is an indiscernible sequence from $\frak{f}(\M)_{x'}\times \M_{x\setminus x'}$ such that
\begin{enumerate}
\item $I = I_1+(c)+I_2$ where $I_1$ and $I_2$ are infinite, and $(a_i)_{i\in I_1+I_2}$ is $b_1b_2$-indiscernible, and
\item $\frak{f}\big(g(a_i,b_1)\big) = h(a_i,b_2)$ for every $i\in I_1+I_2$, 
\end{enumerate}
then $\frak{f}\big(g(a_c,b_1)\big) = h(a_c,b_2)$.
\end{enumerate}
Then $T(\mathfrak{f})$ is distal.
\end{distalcriterion}

\section{$I$-ordered $C$-vector spaces}
\label{IorderedCvectorspaces}

\noindent
\emph{In this section $C$ is an ordered field and $I$ is a nonempty index set.} Later we will consider the case where $|I|=2$, but the presentation would not simplify much if we were to restrict ourselves to that special case.

\begin{definition}
An \textbf{$I$-ordered $C$-vector space} is a $C$-vector space $G$ equipped with an $I$-ordering $(<_i)_{i\in I}$ such that $G$ is an ordered $C$-vector space with respect to each $<_i$. In other words, $G$ is a vector space over $C$, and for every $\lambda\in C$, $x\in G$ and $i\in I$:
\[
\lambda>0 \ \& \ x>_i0 \quad \Rightarrow \quad \lambda x>_i0.
\]
\end{definition}

\medskip\noindent
Let $G,G'$ be two $I$-ordered $C$-vector spaces. We call $G$ an \textbf{$I$-ordered $C$-subspace} of $G'$, or $G'$ an \textbf{extension} of $G$, if, as $C$-vector spaces, $G$ is a subspace of $G'$ and the $I$-ordering on $G$ agrees with the induced $I$-ordering from $G'$ (notation: $G\subseteq G'$). An \textbf{embedding} of $I$-ordered $C$-vector spaces is an embedding $j:G\to G'$ of $C$-vector spaces such that for all $x\in G$ and $i\in I$, if $x>_i0$ in $G$, then $j(x)>_i0$ in $G'$.

\begin{lemma}\label{I-ordered embedding}
Let $G$ be an $I$-ordered $C$-vector space and $\big((P_i,Q_i)\big)_{i\in I}$ a polycut in $G$. Then there is an $I$-ordered $C$-vector space $G'$ extending $G$ and an element $b\in G'$ such that
\begin{enumerate}
\item for any $I$-ordered $C$-vector space extension $G^*$ of $G$ and an element $b^*\in G^*$ which realizes the polycut $\big((P_i,Q_i)\big)_{i\in I}$, there is an embedding $G'\to G^*$ over $G$ of $I$-ordered $C$-vector spaces which sends $b$ to $b^*$.
\end{enumerate}
Furthermore, given any $I$-ordered $C$-vector space $G'$ extending $G$ and $b\in G'$ satisfying (1) above, we have
\begin{enumerate}
\setcounter{enumi}{1}
\item $b$ realizes the polycut $\big((P_i,Q_i)\big)_{i\in I}$, and
\item $G' = G\oplus Cb$ (internal direct sum of $C$-vector spaces).
\end{enumerate}
\end{lemma}
\begin{proof}
As a $C$-vector space, let $G' := G\oplus Cb$. For each $i\in I$, extend $<_i$ to the unique ordered $C$-vector space ordering on $G'$ such that $P_i<_ib<_iQ_i$. The universal property then follows by the universal property in~\cite[2.4.16]{ADAMTT}. (2) and (3) are also clear.
\end{proof}

\begin{definition}
We say that the $I$-ordering $(<_i)_{i\in I}$ on an $I$-ordered $C$-vector space $G$ is \textbf{independent} if for every $n$, distinct $i_1,\ldots,i_n\in I$, and for all pairs $a_1,b_1,\ldots,a_n,b_n\in G_{\pm\infty}$, if $a_k<_{i_k}b_k$ for $k=1,\ldots,n$, then
\[
(a_1,b_1)_{i_1}\ \cap \ \cdots \ \cap \ (a_n,b_n)_{i_n} \quad\text{is nonempty.}
\]
\end{definition}

\noindent
Let $\mathcal{L}_{C,I}$ be the natural language of $I$-ordered $C$-vector spaces:
\[
\mathcal{L}_{C,I} \ = \ \big\{0,+,(\lambda_c)_{c\in C}\big\} \cup \{<_i:i\in I\}.
\]
\noindent
Let $T_{C,I}$ be the $\mathcal{L}_{C,I}$-theory whose models are precisely the $I$-ordered $C$-vector spaces $G$ such that the orderings $(<_i)_{i\in I}$ on $G$ are independent.  By applying Lemma~\ref{I-ordered embedding} iteratively starting with the trivial $I$-ordered $C$-vector space with underlying set $\{0\}$, we can construct a model of $T_{C,I}$ and thus $T_{C,I}$ is consistent. Note that since $I$ is nonempty, models of $T_{C,I}$ are necessarily infinite.

\begin{prop}
\label{TCIQE}
The $\mathcal L_{C,I}$-theory $T_{C,I}$ admits quantifier elimination and is complete.
\end{prop}
\begin{proof}
Let $G$ and $G^*$ be models of $T_{C,I}$ and suppose $H\subseteq G$ is a proper $\mathcal{L}_{C,I}$-substructure of $G$.
Furthermore, suppose $G^*$ is $|H|^+$-saturated and $i:H\to G^*$ is an embedding of $\mathcal{L}_{C,I}$-structures. For quantifier elimination, it suffices to find $b\in G\setminus H$ such that $i$ extends to an embedding $H+Cb\to G^*$ (e.g., see~\cite[B.11.10]{ADAMTT}).

Take $b\in G\setminus H$ and let $\big((P_i,Q_i)\big)_{i\in I}$ be the unique polycut in $H$ realized by $b$. Then the image $\big(i(P_i),i(Q_i)\big)_{i\in I}$ determines a partial type in $G^*$ over the parameter set $i(H)$:
\[
i(P_i)<x<i(Q_i) \quad\text{for every $i\in I$.}
\]
As the orderings on $G^*$ are independent and $G^*$ is $|H|^+$-saturated, we may take $b^*\in G^*$ realizing this partial type. Then by Lemma~\ref{I-ordered embedding}, there is an embedding $H+Cb\to G^*$ which extends $i$ and sends $b$ to $b^*$.

%
%
%

Completeness follows from quantifier elimination and the fact that the trivial $I$-ordered $C$-vector space with a single element embeds into every model of $T_{C,I}$
\end{proof}

\begin{cor}
\label{TCIdistal}
$T_{C,I}$ is distal.
\end{cor}
\begin{proof}
By Proposition~\ref{TCIQE} and Lemma~\ref{equivdistaldefs}, it suffices to show each quantifier-free $\mathcal{L}_{C,I}$-formula $\varphi(x;y)$ is distal. Every atomic formula and negated atomic formula is in a reduct to $\big\{0,+,(\lambda_c)_{c \in C},<_{i}\!\big\}$ for some $i \in I$. Every such reduct is an ordered $C$-vector space, and hence by Example \ref{o-mindist} is distal. Therefore, each formula $\varphi(x;y)$ is equivalent to a boolean combination of distal formulas, hence is distal.
\end{proof}

\section{Hamel spaces}\label{Hamel spaces}
\noindent
\emph{Throughout this section, $C$ is an ordered field.}

\subsection{Hamel valuations} Consider a $2$-ordered $C$-vector space $G$. We denote the two orderings on $G$ by $<_0$ and $<_1$, thinking of $<_0$ as the ``original'' ordering, and $<_1$ as the ``auxiliary'' ordering. 

\medskip\noindent
A \textbf{Hamel valuation} on $G$ is a map $v:G\to G_{\infty}$ which satisfies the following:
\begin{enumerate}
\item $v:G\to G_{\infty}$ is a (non-surjective) convex valuation which makes $G$ a valued vector space over $C$ with respect to the ordering $<_1$ on the vector space and the ordering $<_0$ on the value set, i.e., for all $x,y\in G$ and $\lambda\in C^{\times}$:
\begin{enumerate}
\item $v(x)=\infty$ iff $x=0$;
\item $v(x+y)\geq_0 \min_{0}\big(v(x),v(y)\big)$;
\item $v(\lambda x) = v(x)$; and
\item if $0<_1x<_1y$, then $vx\geq_0 vy$.
\end{enumerate}
\item (Idempotence) $vx = v(vx)$.
\item (Positivity) $vx >_1 0$.
\end{enumerate}

\noindent
A \textbf{Hamel space (over $C$)} is a pair $(G,v)$ where $G$ is a $2$-ordered $C$-vector space, and $v$ is a Hamel valuation on $G$. 

\begin{definition}
Let $(G,v)$ be a Hamel space. We say that $(G,v)$ is \textbf{independent} if the orderings $<_0$ and $<_1$ on $G$ are independent. We say that $(G,v)$ is \textbf{dense} if for every $a,b\in G$, if $a<_0b$, then there is $c\in G$ such that $a<_0vc<_0b$.
\end{definition}

\noindent
Given an independent Hamel space $(G,v)$, the value set $v(G)$ will not be dense in $(G,<_1)$. Indeed, given $x\in G^{\neq}$, we have $v(G)\cap (vx,2vx)_1 = \emptyset$. However, the set $G\setminus v(G)$ is dense in both orderings:

\begin{lemma}
Suppose $(G,v)$ is an independent Hamel space. Then for every $x_0,y_0,x_1,y_1\in G$ such that $x_0<_0y_0$ and $x_1<_1y_1$, there is $z\in G$ such that $z\neq vz$ and $x_0<_0z<_0y_0$ and $x_1<_1z<_1y_1$. 
\end{lemma}
\begin{proof}
By independence, there is $z'\in G$ such that $x_0<_0z'<_0y_0$ and $x_1<_1z'<_1y_1$. If $z'\neq vz'$, then $z:= z'$ will work. Otherwise, necessarily $z'=vz'>_10$. Applying independence again, we get $z\in G$ such that $z'<_0z<_0y_0$ and $z'<_1z<_1\min_{1}(2z', y_1)$. By convexity, we have $vz = vz'\neq z$, so this $z$ works.
\end{proof}

\noindent
In general, the set $v(G^{\neq})$ in a Hamel space $(G,v)$ will not span $G$ as a $C$-vector space. However, $v(G^{\neq})$ will always be $C$-linearly independent:

\begin{lemma}
\label{vGindependent}
Suppose $(G,v)$ is a Hamel space. Then $v(G^{\neq})$ is $C$-linearly independent.
\end{lemma}
\begin{proof}
Suppose $g_1,\ldots,g_n\in v(G^{\neq})$ are such that $g_1<_0\cdots<_0g_n$. Take $\lambda_1,\ldots,\lambda_n\in C^{\times}$. Then $v(\lambda_ig_i) = vg_i = g_i$ for all $i=1,\ldots,n$, and so $v(\lambda_ig_i)\neq v(\lambda_jg_j)$ for $i\neq j$. Thus
\[
v\big(\textstyle\sum_{i=1}^n\lambda_ig_i\big) \ = \ \min_0\big\{v(\lambda_ig_i):i=1,\ldots,n\big\} \ = \ \min_0 \{g_i: i=1,\ldots,n\} \ = \ g_1 \ \neq \ \infty.
\]
In particular, $\sum_{i=1}^n\lambda_ig_i\neq 0$.
\end{proof}

\noindent
The following will be used in our proof of Theorem~\ref{THamdistal}, namely to verify condition (2) in Distal Criterion~\ref{distal_multi}, with $(G,v)$ playing the role of ``$(\bm{N},\frak{f})$''. We actually prove something more general:

\begin{prop}
\label{mainextprop}
Suppose $(G,v)$ is a Hamel space and $G_0\subseteq G$ is a $C$-linear subspace of $G$. Given $c_1,\ldots,c_m\in G\setminus G_0$, we have
\[
\#\big(v(G_0+\textstyle\sum_{i=1}^mCc_i)\setminus v(G_0)\big) \ \leq \ m.
\]
In particular, there is $n\leq m$ and distinct
\[
\textstyle d_1,\ldots,d_n \ \in \ v(G_0+\sum_{i=1}^mCc_i) \setminus v(G_0)
\]
such that
\[
\textstyle v(G_0+\sum_{i=1}^mCc_i) \ \subseteq  \ v(G_0)\cup\{d_1,\ldots,d_n\}.
\]
\end{prop}
\begin{proof}
Assume towards a contradiction that there are $m+1$ distinct $d_1,\ldots,d_{m+1} \ \in \ v(G_0+\sum_{i=1}^mCc_i) \setminus v(G_0)$. For each $j=1,\ldots,m+1$, let $e_j\in G_0+\sum_{j=1}^{m}Cc_i$ such that $ve_j = d_j$. We claim that $e_1,\ldots,e_{m+1}$ are $C$-linearly independent over $G_0$. This follows from the fact that for $g\in G_0$, and $\lambda_1,\ldots,\lambda_{m+1}\in C$, we have
\[
\textstyle v\big(g+\sum_{j=1}^{m+1}\lambda_je_j\big) \ = \ \min_0\big(\{d_j: \lambda_j\neq 0\}\cup \{vg\}\big).
\]
Thus $\dim_C(G_0+\sum_{i=1}^mCc_i)/G_0\geq m+1$, a contradiction.
\end{proof}

\subsection{Extensions of Hamel spaces} \emph{In this subsection $(G,v)$ and $(G',v')$ are Hamel spaces.} We call $(G,v)$ a \textbf{Hamel subspace} of $(G',v')$, or $(G',v')$ an \textbf{extension} of $(G,v')$, if $G\subseteq G'$ as $2$-ordered $C$-vector spaces, and for all $x\in G$, $v(x) = v'(x)$; notation: $(G,v)\subseteq (G',v')$. An \textbf{embedding} $i:(G,v)\to (G',v')$ of Hamel spaces is an embedding $i:G\to G'$ of the underlying $2$-ordered $C$-vector spaces such that for all $x\in G$, $i\big(v(x)\big) = v'\big(i(x)\big)$.

\begin{lemma}[Growing the value set]
\label{valuesetextension}
Suppose $P\subseteq G$ is a cut in $(G,<_0)$. Then there is an extension $(G',v')$ of $(G,v)$ and an element $h\in G'$ such that
\begin{enumerate}
\item $h=v'(h)$,
\item $P<_0 h<_0G\setminus P$, and
\item given any embedding $i:(G,v)\to (G^*,v^*)$ and an element $h^*\in G^*$ such that 
\[
i(P)<_0h^*=v^*h^*<_0i(G\setminus P),
\]
 there is an extension of $i$ to an embedding $(G',v')\to (G^*,v^*)$ which sends $h$ to $h^*$.
\end{enumerate}
\end{lemma}
\begin{proof}
First, we will define the polycut over $G$ that such an element $h$ must realize. Set $P_0:=P$, $Q_0:= G\setminus P$,
\[
P_1 \ := \  \{g\in G: g\leq_1 0\} \cup \{g\in G: vg\in Q_0\},
\]
and $Q_1:= G\setminus P_1$. Let $G' := G+Ch$ be the extension of $G$ of $2$-ordered $C$-vector spaces given in Lemma~\ref{I-ordered embedding} for the polycut $\big((P_i,Q_i)\big)_{i=1,2}$. Next, define the map $v':G'\to G'_{\infty}$ by
\[
v'(g+ch) \ := \ \begin{cases}
v(g) & \text{if $vg\in P_0$ or $c=0$} \\
h & \text{if $vg\not\in P_0$ and $c\neq 0$,}
\end{cases}
\]
for $g\in G$ and $c\in C$.
It is easily checked that $(G',v')$ is an extension of $(G,v)$ with the desired universal property.
\end{proof}

\begin{lemma}[Immediate extension]
\label{immediateextension}
Suppose $P\subseteq G$ is a cut in $(G,<_0)$ and $(b_{\rho})$ is a divergent pc-sequence in $G$. Then there is an extension $(G',v')$ of $(G,v)$ and an element $h\in G'$ such that
\begin{enumerate}
\item $b_{\rho}\leadsto h$,
\item $P<_0 h<_0 G\setminus P$, and
\item given any embedding $i:(G,v)\to (G^*,v^*)$ and an element $h^*\in G^*$ such that
\[
i(b_{\rho})\leadsto h^* \quad\text{and}\quad i(P)<_0 h^*<_0 i(G\setminus P),
\]
there is an extension of $i$ to an embedding $(G',v')\to (G^*,v^*)$ which sends $h$ to $h^*$.
\end{enumerate}
\end{lemma}
\begin{proof}
Let $G' := G+Ch$ be a $C$-vector space extension of the underlying $C$-vector space of $G$. Next, we equip $G'$ with the unique valuation $v'$ which makes $(G',v')$ an immediate extension of $(G,v)$ (as valued vector spaces over $C$), such that $b_{\rho}\leadsto h$; see~\cite[2.3.1]{ADAMTT}. Then by~\cite[2.4.20]{ADAMTT}, there is just one ordering $<_1$ on $G'$ which extends $<_1$ on $G$ which makes $(G',<_1)$ an ordered $C$-vector space and $v'$ a convex valuation with respect to $<_1$. We equip $G'$ with this ordering. Finally, by~\cite[2.4.16]{ADAMTT} there is a unique ordering $<_0$ on $G'$ which extends $<_0$ on $G$ such that $P<_0h<_0 G\setminus P$; we also equip $G'$ with this ordering. It is easily checked that equipped with these orderings, $(G',v')$ is an extension of $(G,v)$ with the desired universal property.
\end{proof}

\noindent
Given $\alpha\in v(G^{\neq})$, the sets
\[
\overline{B}(\alpha) \ := \ \{g\in G: vg\geq_0 \alpha\},\quad\text{and} \quad
B(\alpha) \ := \ \{g\in G: vg>_0 \alpha\}
\]
are convex ordered $C$-vector spaces with respect to the $<_1$-ordering. As $\overline{B}(\alpha)\supseteq B(\alpha)$, the $<_1$-ordering induces an ordering on the quotient $G(\alpha):= \overline{B}(\alpha)/B(\alpha)$, giving it a natural structure as an ordered $C$-vector space. 
Moreover, given an embedding $i:(G,v)\to (G',v')$ and $\alpha\in v(G^{\neq})$, $i$ induces a natural ordered $C$-vector space embedding $i:G(\alpha)\to G'(i\alpha)$.

\medskip\noindent
We define an \textbf{$\alpha$-cut} to be a subset $P\subseteq \overline{B}(\alpha)$ such that
\begin{enumerate}
\item $P$ is downward closed in $(\overline{B}(\alpha),<_1)$, and
\item for all $x,y\in \overline{B}(\alpha)$, if $x-y\in B(\alpha)$, then $x\in P$ iff $y\in P$.
\end{enumerate}
In other words, an $\alpha$-cut is essentially a lift of a cut in the ordered $C$-vector space $G(\alpha)$.

\begin{lemma}[Growing a quotient space]
\label{growingquotientspace}
Suppose $\alpha\in v(G^{\neq})$, $P$ is an $\alpha$-cut, and $P'\subseteq G$ is a cut in $(G,<_0)$. Then there is an extension $(G',v')$ of $(G,v)$ and an element $h\in G'$ such that
\begin{enumerate}
\item $P'<_0h<_0 G\setminus P'$,
\item $vh = \alpha$, 
\item $P<_1 h<_1 \overline{B}(\alpha)\setminus P$, and
\item given any embedding $i:(G,v)\to (G^*,v^*)$ and an element $h^*\in G^*$ such that
\begin{enumerate}
\item $i(P')<_0 h^*< i(G\setminus P')$,
\item for all $g\in G$ and $c\in C$,
\[
v^*\big(i(g)+ch^*\big) \ := \ \begin{cases}
\min_0\big(v^*i(g), i(\alpha)\big) & \text{if $c\neq 0$,} \\
v^*i(g) & \text{otherwise, and} \\
\end{cases}
\]
\item $i(P)<_1 h^*<_1 \overline{B}(\alpha)\setminus P$,
\end{enumerate}
there is an extension of $i$ to an embedding $(G',v')\to (G,v)$ which sends $h$ to $h^*$.
\end{enumerate}
\end{lemma}
\begin{proof}
First, we will define the polycut over $G$ that such an element $h$ must realize. Set $P_0:= P', Q_0:= G\setminus P_0$,
\[
P_1 \ := \ \{g\in G : g<_1 \overline{B}(\alpha)\} \cup P,
\]
and $Q_1 := G\setminus P_1$. Let $G':= G+Ch$ be the extension of $G$ of $2$-ordered $C$-vector spaces given in Lemma~\ref{I-ordered embedding} for the polycut $\big((P_i,Q_i)\big)_{i=1,2}$. Next, define the map $v':G'\to G_{\infty}'$ by
\[
v'(g+ch) \ := \ \begin{cases}
\min_0 (vg, \alpha) &\text{if $c\neq 0$,} \\
vg & \text{otherwise,}
\end{cases}
\]
for $g\in G$ and $c\in C$. It is easily checked that $(G',v')$ is an extension of $(G,v)$ with the desired property.
\end{proof}

\subsection{Model theory of Hamel spaces} Now let $\mathcal{L}_{\text{Ham}}$ be the natural language of Hamel spaces over $C$, i.e.,
\[
\mathcal{L}_{\text{Ham}} \ := \ \mathcal{L}_{C,2} \cup \{v,\infty\} \ = \ \{0,+,(\lambda_c)_{c\in C},<_0,<_1,v,\infty\}.
\]
We consider a Hamel space $(G,v)$ as an $\mathcal{L}_{\text{Ham}}$-structure with underlying set $G_{\infty}$ and the obvious interpretation of the symbols in $\mathcal{L}_{\text{Ham}}$, with $\infty$ as a default value:
\[
g+\infty \ = \ \infty+g \ = \ \lambda_c(\infty) \ = \ v(\infty) \ = \ \infty+\infty \ = \ \infty
\]
for all $g\in G$ and $c\in C$. We let $T_{\text{Ham}}$ be the $\mathcal{L}_{\text{Ham}}$-theory whose models are the independent and dense Hamel spaces over $C$.

\begin{lemma}
$T_{\text{Ham}}$ is consistent.
\end{lemma}
\begin{proof}
We claim that $T_{\text{Ham}}$ has a model. One way to construct a model is to start with the trivial Hamel space with underlying set $\{0,\infty\}$, and then iteratively apply Lemma~\ref{valuesetextension} for denseness and Lemma~\ref{growingquotientspace} for independence.

Alternatively, one can consider the $C$-linear space, $G = \bigoplus_{\lambda<|C|^+}Ce_{\lambda}$, equipped with the lexicographic ordering $<_0$. By cofinality reasons, one can then recursively construct a $C$-vector space basis $H\subseteq G$ which is dense in the $<_0$-ordering.
Then one can define $v:G\to G_{\infty}$ and $<_1$ on $G$ in an analogous way to the structure $(\R,+,<,H)$ from the Introduction. Then $(G,v)$ together with these two orderings will be a model of $T_{\text{Ham}}$.
\end{proof}

\begin{theorem}
\label{THamQE}
The $\mathcal{L}_{\text{Ham}}$-theory $T_{\text{Ham}}$ admits quantifier elimination and is complete.
\end{theorem}
\begin{proof}
Let $(G,v)$ and $(G^*,v^*)$ be models of $T_{\text{Ham}}$ and suppose $(H,v)\subseteq (G,v)$ is a proper $\mathcal{L}_{\text{Ham}}$-substructure of $(G,v)$.
Furthermore, suppose $(G^*,v^*)$ is $|H|^+$-saturated and $i:(H,v)\to (G^*,v^*)$ is an embedding of $\mathcal{L}_{\text{Ham}}$-structures. For quantifier elimination, it suffices to find $h\in G\setminus H$ such that $i$ extends to an embedding $(H+Ch,v)\to (G^*,v^*)$ (e.g., see~\cite[B.11.10]{ADAMTT}). We consider three cases:

{\bf Case 1:} \emph{There is $h\in v(G)\setminus v(H)$.} Choose such an $h\in G$. Set $P:= \{g\in H: g<_0 h\}$ and $Q:= H\setminus P$. By saturation and denseness of $(G^*,v^*)$, there is $h^*\in v^*(G^*)\subseteq G^*$ such that $i(P)<_0h^*<_0i(Q)$. 
By Lemma~\ref{valuesetextension}, $i$ extends to an embedding $(H+Ch,v)\to (G^*,v^*)$ which sends $h$ to $h^*$.

{\bf Case 2:} \emph{There is $h\in G\setminus H$ such that $v(h-H)$ does not have a largest element.} Choose such an $h\in G$. Take a well-indexed sequence $(b_{\rho})$ in $H$ such that $\big(v(h-b_{\rho})\big)$ is strictly increasing and cofinal in $v(h-H)$. Then $(b_{\rho})$  is a divergent pc-sequence in $H$ such that $b_{\rho}\leadsto h$. Set $P:= \{g\in H: g<_0 h\}$ and $Q:= H\setminus P$. Then by saturation and independence of $(G^*,v^*)$ there is $h^*\in G^*$ such that $i(b_{\rho})\leadsto h^*$ and $i(P)<_0h^*<_0 i(Q)$. 
By Lemma~\ref{immediateextension}, $i$ extends to an embedding $(H+Ch,v)\to (G^*,v^*)$ which sends $h$ to $h^*$.

{\bf Case 3:} \emph{There is $h\in G$ such that $v(h)\in H$ and there is no $g\in H$ such that $v(h-g)>_0v(h)$.} Choose such an $h\in G$. Set $P':= \{g\in H: g<_0 h\}$, a cut in $(H,<_0)$. Furthermore, define
\[
P \ := \ \{g\in H: vg\geq_0 vh \text{ and } g\leq_1 h\},
\]
which is an $vh$-cut in $H$ by the assumption on $h$. 
Next, in the quotient space $G^*(ivh)$, pick an element $\bar{h}$ such that $i(P)+B(ivh)<_1\bar{h}<_1 i(H\setminus P)+B(ivh)$, which can be done by saturation of the ordered $C$-vector space $G^*(ivh)$, an interpretable structure in $(G^*,v^*)$. By independence and saturation of $(G^*,v^*)$, there is an element $h^*\in \overline{B^*}(ivh)\subseteq G^*$ such that $h^*+B^*(ivh) = \bar{h}$ and $i(P')<_0h^*<_0 i(H\setminus P)$ (note that this also uses the fact that $B^*(ivh)$ has at least two elements, a consequence of denseness and independence).
Then by Lemma~\ref{growingquotientspace}, $i$ extends to an embedding $(H+Ch,v)\to (G^*,v^*)$ which sends $h$ to $h^*$.

Completeness follows from quantifier elimination and the fact that the trivial Hamel space with underlying set $\{0,\infty\}$ embeds into every model of $T_{\text{Ham}}$.
\end{proof}

\section{Distality for Hamel spaces}
\label{DistalityHamelSpacesSection}

\noindent
In this section we prove the main result of this paper:

\begin{thm}
\label{THamdistal}
$T_{\text{Ham}}$ is distal.
\end{thm}

\noindent
This has the following consequences, also of interest:

\begin{cor}
$T_{\text{Ham}}$ has the non-independence property (NIP).
\end{cor}
\begin{proof}
It is well-known that distality implies NIP; e.g., see~\cite[Proposition 2.8]{GehretKaplan} for a proof.
\end{proof}

\begin{cor}
No model of $T_{\text{Ham}}$ interprets an infinite field of positive characteristic.
\end{cor}
\begin{proof}
See~\cite[Corollary 6.3]{ChernikovStarchenko}.
\end{proof}

\noindent
\emph{In the rest of this section $\M$ is a monster model of $T_{\text{Ham}}$ with underlying set $G_{\infty}$.}

\subsection{Indiscernible lemmas} In this subsection we will prove the main lemmas involving indiscernible sequences in $G_{\infty}$ that we need for verifying condition (3) in Distal Criterion~\ref{distal_multi}. \emph{We assume in this subsection that $I = I_1+(c)+I_2$ is an ordered index set with infinite $I_1$ and $I_2$, and $i,j,k$ range over $I$.}

\begin{lemma}
\label{nonconstantconstant}
Let $(a_i)_{i\in I}$ be a nonconstant indiscernible sequence from $G$ and suppose $(b,b')\in G\times v(G)$ is such that $(a_i)_{i\in I_1+I_2}$ is $bb'$-indiscernible. If $v(a_i-b) = b'$ for all $i\neq c$, then $v(a_c-b) = b'$.
\end{lemma}
\begin{proof}
Without loss of generality we may assume that if $i<j$, then $a_i-b \leq_1 a_j-b$. We may also assume that $0<_1a_i-b$ for all $i$. Now for $i \in I_1$ and $j \in I_2$ we have $v(a_i-b)=v(a_j-b)$ as well as $0<_1a_i-b\leq_1a_c-b\leq_1a_j-b$, so as $v$ is convex with respect to the $<_1$-ordering, $v(a_c-b)=b'$ as well. 
\end{proof}

\begin{lemma}
\label{nonconstantnonconstant}
Let $(a_ia_i')_{i\in I}$ be an indiscernible sequence from $G\times v(G)$ such that $(a_i)$ and $(a_i')$ are each nonconstant, and suppose $b\in G$ is such that $(a_ia_i')_{i\in I_1+I_2}$ is $b$-indiscernible. If $v(a_i-b) = a_i'$ for all $i\neq c$, then $v(a_c-b) = a_c'$.
\end{lemma}
\begin{proof}
Without loss of generality, assume that $(a_i')_{i\in I}$ is strictly increasing in the $<_0$-ordering. Let $i,j \in I_1+I_2$ with $i<j$. Then 
\[
v(a_i-a_j) \ = \ v\big((a_i-b)+(b-a_j)\big) \ = \ \min_0(a_i',a_j') \ = \ a_i'.
\]
By indiscernibility, for $j\in I_2$, $v(a_c-a_j)=a_c'$ and so 
\[
v(a_c-b) \ =\ v\big((a_c-a_j)+(a_j-b)\big) \ = \ \min_0(a_c',a_j') \ = \ a_c' \qedhere
\]
\end{proof}

\noindent
In the next two lemmas $\mathcal{L}:= \{0,+,(\lambda_c)_{c\in C}, <_0, <_1,\infty\}\subseteq \mathcal{L}_{\text{Ham}}$.

\begin{lemma}
\label{gconstantseq}
Let $g,h$ be $\mathcal{L}$-terms of arities $n+k$ and $m+l$ respectively with $m\leq n$, $b_1\in \M^k$, $b_2\in v(\M)^l$, $(a_i)_{i\in I}$ be an indiscernible sequence from $v(\M)^m\times \M^{n-m}$ such that
\begin{enumerate}
\item $(a_i)_{i\in I_1+I_2}$ is $b_1b_2$-indiscernible,
\item $v\big(g(a_i,b_1)\big) = h(a_i,b_2)$ for every $i\neq c$, and
\item $\big(g(a_i,b_1)\big)_{i\in I_1+I_2}$ is a constant sequence.
\end{enumerate}
Then $v\big(g(a_c,b_1)\big) = h(a_c,b_2)$.
\end{lemma}
\begin{proof}
This is routine and left to the reader. See the proof of~\cite[Lemma 4.3]{GehretKaplan}.
\end{proof}

\begin{lemma}
\label{nonconstantsimplification}
Let $h(x,y)$ be an $\mathcal{L}$-term of arity $m+n$, $b\in \M^n$, and $(a_i)_{i\in I}$ an indiscernible sequence from $v(\M)^m$, with $a_i = (a_{i,1},\ldots,a_{i,m})$. Assume that $h(a_i,b)\in v(\M)$ for infinitely many $i$. Then one of the following is true:
\begin{enumerate}
\item $h(a_i,b) = \infty$ for every $i$;
\item there is $\beta\in v(G^{\neq})$ such that $h(a_i,b) = \beta$ for every $i$;
\item there is $l\in \{1,\ldots,m\}$ such that $h(a_i,b) = a_{i,l}$ for every $i$.
\end{enumerate}
\end{lemma}
\begin{proof}
This is an exercise in simplification and bookkeeping which mimics the proof of~\cite[Lemma 4.4]{GehretKaplan}, except that it uses the fact  that the value set $v(G^{\neq})$ is a $C$-linearly independent subset of $G$ (Lemma~\ref{vGindependent}).
\end{proof}

%
%
%

\subsection{Proof of Theorem~\ref{THamdistal}} In this subsection we prove Theorem~\ref{THamdistal} by verifying the hypotheses of Distal Criterion~\ref{distal_multi}. In the language of~\ref{distal_multi}, the role of $T$ will be played by the reduct $T:= T_{\text{Ham}}\upharpoonright\mathcal{L}$, where $\mathcal{L}:= \mathcal{L}_{\text{Ham}}\setminus\{v\} = \{0,+,(\lambda_c)_{c\in C}, <_0, <_1,\infty\}$. The $\mathcal{L}$-theory $T$ is bi-interpretable with the $\mathcal{L}_{C,2}$-theory $T_{C,2}$. Indeed, $T$ is essentially the same thing as $T_{C,2}$, except that $T$ has an extra point $\infty$ at infinity with respect to both orders which serves as a default value with respect to the $C$-vector space structure. As distality is preserved under bi-interpretability, by Corollary~\ref{TCIdistal} we have that $T$ is distal.

\medskip\noindent
In the language of~\ref{distal_multi}, we also construe $T_{\text{Ham}}$ as $T_{\text{Ham}} = T(v)$, and in particular, $\mathcal{L}_{\text{Ham}} = \mathcal{L}(v)$. Since $T_{\text{Ham}}$ has quantifier elimination (Theorem~\ref{THamQE}), this verifies condition (1) in~\ref{distal_multi}. Condition (2) in~\ref{distal_multi} follows from Proposition~\ref{mainextprop}.

\medskip\noindent
Finally we will verify condition (3) in~\ref{distal_multi}. Let $f,g$ be $\mathcal{L}$-terms of arities $n+k$ and $m+l$ respectively, with $m\leq n$, $b_1\in \M^k$, $b_2\in v(\M)^l$, $(a_i)_{i\in I}$ be an indiscernible sequence from $v(\M)^m\times \M^{n-m}$ such that
\begin{enumerate}[(a)]
\item $I = I_1+(c)+I_2$ with infinite $I_1$ and $I_2$, and $(a_i)_{i\in I_1+I_2}$ is $b_1b_2$-indiscernible, and
\item $v\big(g(a_i,b_1)\big) = h(a_i,b_2)$ for every $i\in I_1+I_2$.
\end{enumerate}
Our job is to show that $v\big(g(a_c,b_1)\big) = h(a_c,b_2)$. We have several cases to consider:

\textbf{Case 1:} \emph{$\big(g(a_i,b_1)\big)_{i\in I_1+I_2}$ is a constant sequence.} In this case, $v\big(g(a_c,b_1)\big) = h(a_c,b_2)$ follows from Lemma~\ref{gconstantseq}.

For the remainder of the proof, we assume that \emph{$\big(g(a_i,b_1)\big)_{i\in I_1+I_2}$ is not a constant sequence}. In particular, the symbol $\infty$ does not play a non-dummy role in $g(a_i,b_1)$, so the $\mathcal{L}$-term $g(x,y)$ is essentially a $C$-linear combination of its arguments. By grouping these $C$-linear combinations, we get $b\in G$, and a nonconstant indiscernible sequence $(a_i')_{i\in I}$ from $\M$ such that
\begin{enumerate}[(a)]
\setcounter{enumi}{2}
\item $g(a_i,b_1) = a_i'-b$ for every $i\in I$,
\item $(a_i,a_i')_{i\in I}$ is an indiscernible sequence from $v(\M)^m\times \M^{n-m+1}$,
\item $(a_ia_i')_{i\in I+I_1}$ is $b_1b_2b$-indiscernible, and
\item $v(a_i'-b) = h(a_i,b_2)$ for every $i\in I_1+I_2$.
\end{enumerate}
We now must show that $v(a_c'-b) = h(a_c,b_2)$. Since $h(a_i,b_2)\in v(\M)$ for all $i\in I_1+I_2$, by Lemma~\ref{nonconstantsimplification} we have three more cases to consider:

\textbf{Case 2:} \emph{$h(a_i,b_2) = \infty$ for every $i\in I$.} In this case, we have $v(a_i'-b) = \infty$ for every $i\in I_1+I_2$, so $a_i' = b$ for every $i\in I_1+I_2$. Thus $(a_i')_{i\in I}$ is a constant sequence and thus $v(a_c'-b) = \infty$ as well.

\textbf{Case 3:} \emph{There is $\beta\in v(G^{\neq})$ such that $h(a_i,b_2) = \beta$ for every $i\in I$.} This case follows from Lemma~\ref{nonconstantconstant}.

\textbf{Case 4:} \emph{There is $l\in \{1,\ldots,m\}$ such that $h(a_i,b_2) = a_{i,l}$ for every $i\in I$.} This case follows from Lemma~\ref{nonconstantnonconstant}.

\section{Connection to dense pairs and independent sets}
\label{ConnectionDensePairs}

\noindent
In \cite{Hieronymi_Nell}, Hieronymi and Nell considered whether certain commonly studied pair structures were distal. These include expansions of o-minimal structures by dense independent sets~\cite{DolichMillerSteinhorn} and dense pairs of o-minimal structures~\cite{DensePairs}. That is, expansions by a dense independent set and expansions by proper, dense elementary substructures. We now show that a model of $T_\text{Ham}$ interprets both the expansion of an ordered $C$-vector space by a dense $C$-independent set and the expansion by a proper, dense elementary substructure. 

\begin{cor}
\label{interpretingthings}
Let $\M$ be a model of $T_\text{Ham}$ with underlying set $G_\infty$. Then there are
\begin{enumerate}
\item a definable, dense, $C$-linearly independent $H \subseteq G$ and
\item a definable $S\subsetneq G$ that is the underlying set of an elementary substructure of $G$ as an ordered $C$-vector space.
\end{enumerate}
Hence $T_\text{Ham}$ interprets a distal expansion for both independent pairs of ordered $C$-vector spaces and dense pairs of ordered $C$-vector spaces. 
\end{cor}
\begin{proof}
Note that $H:=v(G^{\neq})$ is $C$-linearly independent by Lemma~\ref{vGindependent} and $<_0$-dense in $G$, since $\M$ is dense as a Hamel space. Thus the structure $(G;+,<_0,(\lambda_c)_{c \in C},H)$ is an independent pair of ordered $C$-vector spaces.

With $H$ as above, consider the upward-closed subset of the value set:
\[
H_0 \ := \ \{h\in H: h>_0 0\} \cup \{\infty\}.
\]
This yields a certain ``generalized ball'':
\[
S \ := \ \big\{g\in G: v(g)\in H_0\big\}.
\]
$S$ is closed under $C$-linear combinations. Furthermore, as $\M$ is independent, $S$ is dense in $G$. Thus the pair $(G;+,<_0,(\lambda_c)_{c \in C},S)$ is a dense pair of ordered $C$-vector spaces.
\end{proof}

\noindent
Thus a model of $T_\text{Ham}$ interprets distal expansions for both independent pairs and dense pairs of ordered $C$-vector spaces. While this was known previously for dense pairs~\cite{NellMLQ}, this was unknown for independent pairs. In fact, the strategy used for dense pairs relied on manipulating imaginary sorts in the quotient by the substructure. However, independent pairs eliminate imaginaries \cite{DolichMillerSteinhorn}. Thus a new approach was necessary for this case.

\medskip\noindent
Of course, in this paper we restricted our attention to the case where the base o-minimal theory is an ordered vector space. 
The case of expanding a field is also of interest:

\begin{question}
Suppose $\mathfrak{R}$ is an o-minimal expansion of $(\R;+,\cdot)$ and $H\subseteq \R$ is a dense and $\operatorname{dcl}_{\mathfrak{R}}$-independent subset of $\R$.
 Does $(\mathfrak{R},H)$ have a distal expansion?
\end{question}

\noindent
We believe the answer to be yes, and perhaps such a distal expansion can be constructed in a manner similar to our construction here (with a new valuation $v$ and auxiliary ordering $<_1$). However, such an expansion will undoubtedly require stronger results from o-minimality than we used here.

\medskip\noindent
Finally, we would like to point out a somewhat curious observation arising from the proof of Corollary~\ref{interpretingthings}. In our construction, a dense pair shows up as a certain ball in the valuation. Likewise, an independent set $H$ shows up as the value set for the same valuation. This suggests that, for instance, the two expansions $(\R;+,<,H)$ and $(\R;+,<,\Q)$ of $(\R;+,<)$ are in some sense ``orthogonal'' to each other. Is this just a coincidence or is there something going on here? 
Our instinct says it is the latter and perhaps a more general Hamel space-like construction might serve to formalize the connection between dense pairs and independent sets.




\section{$\operatorname{DP}$-rank}
\label{DPranksection}

\noindent
For the sake of completeness, in this final section we characterize the complexity of the theory $T_{\text{Ham}}$ with regard to the notion of \emph{$\operatorname{dp}$-rank}. Among NIP theories, \emph{$\operatorname{dp}$-rank} provides a finer form of classification of the complexity of a theory. In this scale, \emph{$\operatorname{dp}$-minimal} is the simplest, then followed by \emph{having finite $\operatorname{dp}$-rank}, and then \emph{being strongly dependent}. We show that $T_{\text{Ham}}$ is \emph{not strongly dependent}, which is on the complicated end of the scale. For definitions of these concepts, see~\cite{Usvyatsov} or~\cite[Chapter 4]{SimonNIP}.

\begin{thm}
\label{notstronglydependent}
$T_{\text{Ham}}$ is not strongly dependent. Therefore it is not $\operatorname{dp}$-minimal and does not have finite $\operatorname{dp}$-rank.
\end{thm}

\noindent
It is sufficient to show that $T_{\text{Ham}}$ is not \emph{strong}, a consequence of \emph{strongly dependent} (see~\cite{DolichGoodrick}). For this, we will use the following:

\begin{prop}\cite[2.14]{DolichGoodrick}
\label{DolichGoodrickProp}
Suppose that $\bm{M} = (M;+,<,\ldots)$ is an expansion of a densely-ordered abelian group. Let $\bm{N}$ be a saturated model of $\operatorname{Th}(\bm{M})$, and suppose that for every $\epsilon>0$ in $\bm{N}$ there is an infinite definable discrete set $X\subseteq \bm{N}$ such that $X\subseteq (0,\epsilon)$. Then $\operatorname{Th}(\bm{M})$ is not strong.
\end{prop}

\begin{proof}[Proof of Theorem~\ref{notstronglydependent}]
Let $\bm{N}$ be a saturated model of $T_{\text{Ham}}$. For the purposes of this proof and using Proposition~\ref{DolichGoodrickProp}, we construe $\bm{N} = (N; +,<_1,\ldots)$ as an expansion\footnote{Technically speaking, $\bm{N}$ is only bi-interpretable with an expansion of a densely-ordered abelian group, due to the presence of the element $\infty$. In order to avoid being overly-pedantic, we allow ourselves to be ``sloppy'' and work with essentially the induced structure on $N\setminus\{\infty\}$.} of a densely-ordered abelian group with respect to the $<_1$-ordering. Let $\epsilon>_10$ be such that $\epsilon\neq\infty$. Pick $g\in N$ such that $\infty\neq g=vg>_0 v\epsilon$, which is possible by denseness. It is easily checked that the definable set
\[
X \ := \ \{x\in N : x\neq\infty \ \& \ x = vx \ \& \ x>_0 g\}
\]
is an infinite discrete set (with respect to the order topology induced by $<_1$) such that $X\subseteq (0,\epsilon)_1$. Thus $T_{\text{Ham}}$ is not strong.
\end{proof}

\noindent
In general, among NIP theories there is no clear correlation between distality and $\operatorname{dp}$-rank. We started with a non-distal structure $(\R;+,<,H)$ with is strongly dependent~\cite[2.28]{DolichMillerSteinhorn}, and constructed a distal expansion $(\R;+,<,<_1,v,\infty)$ which is not strongly dependent. Perhaps there is a ``milder'' distal expansion out there:

\begin{question}
Does $(\R;+,<,H)$ admit a strongly dependent distal expansion?
\end{question}

\section*{Acknowledgements}
\noindent
The authors thank Matthias Aschenbrenner, Artem Chernikov, and Philipp Hieronymi for various conversations and correspondences around the topics in this paper, and Julian Ziegler-Hunts for providing feedback on an earlier draft of this paper.
The first author is supported by the National Science Foundation under Award No. 1703709.

\bibliographystyle{amsplain}	
\bibliography{hamel_groups}

\end{document}